\documentclass[10pt, a4paper, reqno]{amsart}

\usepackage[utf8]{inputenc}
\usepackage[T1]{fontenc}
\usepackage{lmodern}
\usepackage[margin=2.5cm]{geometry}
\usepackage{microtype}

\usepackage{amsmath,amssymb,amsthm,mathtools}
\mathtoolsset{showonlyrefs}

\usepackage{braket}

\usepackage{tikz-cd}

\usepackage{xcolor}

\usepackage[colorlinks=true,
            linkcolor=blue,
            citecolor=red,
            urlcolor=blue,
            pdfborder={0 0 0},
            pdftitle={Splitting of Clifford groups associated to finite abelian groups},
            pdfauthor={Cesar Galindo},
            pdfkeywords={Clifford group, finite Heisenberg group, group extension, group cohomology, Weil representation}]{hyperref}

\newtheorem{theorem}{Theorem}[section]
\newtheorem{proposition}[theorem]{Proposition}
\newtheorem{lemma}[theorem]{Lemma}

\theoremstyle{definition}
\newtheorem{definition}[theorem]{Definition}

\theoremstyle{remark}
\newtheorem{remark}[theorem]{Remark}

\newcommand{\Z}{\mathbb{Z}}
\newcommand{\C}{\mathbb{C}}
\newcommand{\HH}{\mathcal{H}}
\newcommand{\cO}{\mathcal{O}}
\newcommand{\F}{\mathbb{F}}
\DeclareMathOperator{\Alt}{Alt}
\DeclareMathOperator{\Map}{Map}
\DeclareMathOperator{\Aut}{Aut}
\DeclareMathOperator{\Inn}{Inn}
\DeclareMathOperator{\Sp}{Sp}
\DeclareMathOperator{\Ps}{Ps}
\DeclareMathOperator{\SL}{SL}
\DeclareMathOperator{\Hom}{Hom}
\DeclareMathOperator{\Sym}{Sym}
\DeclareMathOperator{\id}{id}
\DeclareMathOperator{\Bil}{Bil}

\title{Splitting of Clifford groups associated to finite abelian groups}

\author{C\'esar Galindo}
\address{Departamento de Matem\'aticas, Universidad de los Andes, Carrera 1 No. 18A-12, Edificio H, Bogot\'a 111711, Colombia}
\email{cn.galindo1116@uniandes.edu.co}


\date{}

\begin{document}

\begin{abstract}
The Clifford group associated with a finite abelian group fits into an
extension of the symplectic group by the underlying phase
space.  We prove that this extension splits as a semidirect product if
and only if the order of the underlying abelian group is not divisible
by four.  The obstruction
to splitting is controlled entirely by the 2-primary component and
vanishes precisely when this component is trivial or cyclic of order
two.  This confirms a conjecture of Korbelář and Tolar, extending
their cyclic result to arbitrary finite abelian groups.
\end{abstract}

\maketitle


\section{Introduction}

Clifford operators occur in the stabilizer formalism for quantum
computation and quantum error correction, including constructions of
fault-tolerant gates and stabilizer
codes~\cite{Gottesman1997,Gottesman1998}.  Their role in efficient
simulation is expressed by the Gottesman--Knill
theorem~\cite{AaronsonGottesman2004}, and they are also used in
magic-state distillation~\cite{BravyiKitaev2005}.  Clifford twirling
and unitary $2$-designs are used for fidelity
estimation~\cite{Dankert2009}, while randomized benchmarking is used
to characterize quantum gates~\cite{Knill2008,Magesan2011}.

Finite Clifford groups also arise from irreducible representations of
finite Heisenberg groups.  By the Stone--von Neumann theorem,
symplectic automorphisms act projectively on such a representation,
giving the finite Weil
representation~\cite{Mackey1949,Mackey1958,Prasad2011,Weil1964}.  We
consider the following question for an arbitrary finite abelian group:
when does the projective Clifford group split over the symplectic group
as a semidirect product?

We write $\Z_N\coloneqq\Z/N\Z$.  While the literature often focuses
on systems of~$n$ qubits ($A=(\Z_2)^n$) or on a single qudit of
dimension~$N$ ($A=\Z_N$),
the construction applies to any finite abelian group~$A$;
see also~\cite{Hostens2005,MosesEtAl2026} for Clifford operations in
arbitrary and mixed dimensions.  Metaplectic operators and Weil
representations for finite abelian groups are studied further in
\cite{DuttaPrasad2015,KaiblingerNeuhauser2009}.  The construction uses
the \emph{double}
\[
V_A = A \oplus \widehat{A},
\qquad
\widehat{A}=\Hom(A,U(1)),
\]
which carries the non-degenerate alternating bicharacter
\[
\omega_A\big((a,\chi),(b,\psi)\big)
  = \chi(b)\,\psi(a)^{-1},
\]
making $(V_A,\omega_A)$ a finite symplectic abelian group.  We write
$\Sp(V_A)$ for the group of automorphisms of~$V_A$
preserving~$\omega_A$.  The Heisenberg group~$H(A)$ is the central
extension of~$V_A$ by the circle group determined by the evaluation
pairing, and the projective Clifford group~$C(A)$ is the normalizer
of~$H(A)$ modulo global phases.  Throughout, ``Clifford group'' means
this projective quotient unless explicitly stated otherwise.

After fixing an isomorphism
\[
A\cong \Z_{d_1}\oplus\cdots\oplus\Z_{d_r},
\]
we obtain a tensor-product realization
\[
\C[A]\cong
\C[\Z_{d_1}]\otimes\cdots\otimes\C[\Z_{d_r}]
\cong \C^{d_1}\otimes\cdots\otimes\C^{d_r}.
\]
More explicitly, write $a=(a_1,\dots,a_r)$ and
$\chi=(\chi_1,\dots,\chi_r)$ under the induced decompositions of
$A$ and $\widehat A$.  The global shift and phase operators are
\[
X_a=X_{a_1}\otimes\cdots\otimes X_{a_r},
\qquad
Z_\chi=Z_{\chi_1}\otimes\cdots\otimes Z_{\chi_r},
\]
where, on the $j$-th tensor factor,
$X_{a_j}\ket{b_j}=\ket{a_j+b_j}$ and
$Z_{\chi_j}\ket{b_j}=\chi_j(b_j)\ket{b_j}$; see
Subsection~\ref{subsec:heisenberg-extension} for the definitions of
these operators on $\C[A]$.  Thus $H(A)$ is the Pauli--Heisenberg group of a multipartite system
with local dimensions $d_1,\dots,d_r$, while $C(A)$ is its full
projective Clifford normalizer.  Relative to this tensor decomposition,
$C(A)$ may contain nonlocal Clifford operators as well as tensor
products of local ones.

Conjugation by Clifford operators preserves~$H(A)$ and induces
symplectic automorphisms of~$V_A$.  Conjugation therefore yields the
exact sequence
\begin{equation}\label{eq:extension}
1 \longrightarrow V_A \longrightarrow C(A)
  \longrightarrow \Sp(V_A) \longrightarrow 1.
\end{equation}
Splitting of~\eqref{eq:extension} is detected by its class in
$H^2(\Sp(V_A),V_A)$, which vanishes exactly in the split case.  When
the extension splits, the
symplectic group embeds as a subgroup of the Clifford group via a
group homomorphism $\sigma\colon \Sp(V_A)\to C(A)$, providing a
coherent choice of representative for each symplectic transformation
and yielding an explicit semidirect-product decomposition of every
Clifford operator into a symplectic part and a displacement.  When it
does not, no such global section exists: while individual Clifford
operators can still be described by their symplectic and phase
components, there is no way to choose these components
multiplicatively across the entire symplectic group.

In the cyclic case $A = \Z_N$, the odd and even cases behave quite
differently.  For odd~$N$, Appleby~\cite[Theorem~1]{Appleby2005} and
Gross~\cite[Theorem~3 and Lemma~4]{Gross2006} construct a
semidirect-product description of the projective Clifford group over
its symplectic quotient.  For even~$N$,
Korbe\-l\'a\v{r} and Tolar~\cite{KT23} proved that the extension
fails to split if and only if $4 \mid N$, using a presentation
of~$\SL(2,\Z_N)$ by generators and relations to derive incompatible
constraints on any hypothetical splitting homomorphism.  Motivated by
that result, and by its relation to the work of Bolt, Room,
and Wall~\cite{BoltRoomWall1961a,BoltRoomWall1961b} on Clifford
collineation groups, they conjectured that divisibility of $|A|$ by
$4$ should govern splitting for arbitrary finite abelian groups.  A
recent preprint by Korbel\'a\v{r} and Tolar proves the same splitting
criterion by means of explicit symplectic generators and
relations~\cite{KorbelarTolar2026}.

The purpose of this paper is to confirm this conjecture and to
determine exactly when the extension~\eqref{eq:extension} splits.
Our main result is the following theorem.

\begin{theorem}\label{thm:main}
Let $A$ be a finite abelian group.  The Clifford
extension~\eqref{eq:extension} splits as a semidirect product if and
only if $4 \nmid |A|$.
\end{theorem}

Let $A_2$ denote the $2$-primary component of $A$.  In terms of
primary decomposition, Theorem~\ref{thm:main} states that
\eqref{eq:extension} splits exactly when $A_2=1$ or
$A_2\cong\Z_2$.  Thus the odd-primary components do not affect
splitting.  The case $A_2=1$ includes the odd-dimensional cyclic
splittings of Appleby and Gross, and the theorem extends the cyclic
criterion to arbitrary finite abelian groups.

For the chosen cyclic decomposition, $|A|=d_1\cdots d_r$, so
Theorem~\ref{thm:main} gives the multipartite criterion
\[
4\nmid d_1\cdots d_r.
\]
Thus the extension splits precisely when all local dimensions are
odd, or when exactly one local dimension is congruent to $2$ modulo
$4$ and all the others are odd.  It is
nonsplit as soon as one local dimension is divisible by~$4$, or two
local dimensions are even.  This criterion concerns the existence of
a multiplicative choice of projective Clifford lifts; it does not
assert the existence or absence of nonlocal Clifford gates.

The proof has three main ingredients.  First, we show that the
splitting problem respects primary decomposition: if
$A = B_1 \oplus B_2$ with coprime orders, then the extension for~$A$
splits if and only if the extensions for~$B_1$ and~$B_2$ split.  This
reduces the problem to $p$-primary groups.  Second, for odd primes, we
construct explicit splitting sections using square roots.  Third, for
the prime~$2$, we prove that two basic nonsplitting phenomena account
for the general obstruction.  For the cyclic group $A = \Z_{2^k}$ with
$k \geq 2$, we use the pseudo-symplectic model to derive incompatible
lifting constraints.  For the elementary abelian group $A = (\Z_2)^n$
with $n \geq 2$, the Clifford extension is identified with the
symplectic automorphism extension of a central product
$E_n\circ\Z_4$, and results of Griess~\cite{Griess1973} yield
nonsplitting.  An embedding theorem
then propagates nonsplitting from direct summands to the ambient
group.  The exceptional case is $A = \Z_2$, for which
$\Sp(V_A) \cong \mathfrak{S}_3$ and the extension splits.

The rest of this paper is organized as follows.
Section~\ref{sec:prelim} develops the Heisenberg and Clifford groups
for finite abelian groups and introduces the pseudo-symplectic model
together with the obstruction cocycle.
Section~\ref{sec:decomposition_reduction} proves coprime
functoriality, constructs an explicit splitting when $|A|$ is odd,
and reduces the splitting problem to direct summands.
Section~\ref{sec:nonsplit-cyclic} proves nonsplitting
for~$A=\Z_{2^k}$.
Section~\ref{sec:symplectic-obstruction} treats $A=(\Z_2)^n$ via
the symplectic automorphism extension of a central product built from
an extraspecial $2$-group, using results of Griess.
Finally, Section~\ref{sec:splitting-criterion} completes the proof of
the splitting criterion.

\section{Preliminaries: cohomology and the Clifford extension}\label{sec:prelim}

In this section, we recall the factor-set description of group
extensions and their classification by second cohomology, using
Brown~\cite{Brown1982} as a reference. We then realize the Heisenberg
group as a specific central extension associated with $V_A$, and
classify such extensions using bicharacters.
\subsection{Group extensions and factor sets}\label{subsec:extensions-factor-sets}

Let $G$ and $M$ be groups, with $M$ abelian.  We write the operation
on $M$ multiplicatively throughout this subsection.  An
\emph{extension} of $G$ by $M$ (with $M$ the normal subgroup) is a
short exact sequence
\begin{equation}
1 \longrightarrow M \xrightarrow{\; i \;} E \xrightarrow{\; \pi \;} G \longrightarrow 1.
\end{equation}
To describe the group law of $E$ in terms of $G$ and $M$, we choose a set-theoretic \emph{section} $s\colon G \to E$ such that $\pi(s(g)) = g$ for all $g \in G$, with normalization $s(1_G) = 1_E$.
The set-theoretic section $s$ induces an action of $G$ on $M$ via
conjugation:
\begin{equation}
i(g \cdot m) \coloneqq s(g) i(m) s(g)^{-1}.
\end{equation}
Since $M$ is abelian, this action does not depend on the choice of
section $s$; hence it defines a $G$-module structure on $M$.

A set-theoretic section $s$ need not be a homomorphism; consequently, the product of sections differs from the section of the product by a term in $M$. We define the \emph{factor set} or \emph{2-cocycle} $\beta\colon G \times G \to M$ by the relation
\begin{equation}
s(g)s(h) = i(\beta(g, h)) s(gh).
\end{equation}
Associativity in $E$ implies the \emph{2-cocycle identity}
\begin{equation}
(g \cdot \beta(h, k)) \beta(g, hk) = \beta(g, h) \beta(gh, k).
\end{equation} 
Since the section is normalized, the cocycle is normalized as well: $\beta(1_G,g)=\beta(g,1_G)=1$ for all $g\in G$.
The set of such functions forms an abelian group under pointwise multiplication, denoted by $Z^2(G, M)$.

Conversely, the data $(G,M,\cdot,\beta)$ define a group structure on
$M\times G$ by
\begin{equation}
(m, g) \cdot (n, h) \coloneqq (m \, (g \cdot n) \, \beta(g, h), \, gh).
\end{equation}
The bijection $(m,g)\mapsto i(m)s(g)$ identifies this group with $E$.
We denote this crossed-product structure by $M \times_\beta G$,
reserving $M\rtimes G$ for the split case $\beta=1$.

The choice of section is not unique.  If $s'\colon G\to E$ is another
normalized section, there is a normalized $1$-cochain
$\lambda\colon G\to M$ such that
$s'(g)=i(\lambda(g))s(g)$.  The corresponding cocycle $\beta'$
satisfies
\begin{equation}
\beta'(g, h) = \beta(g, h) \frac{\lambda(g) (g \cdot \lambda(h))}{\lambda(gh)}.
\end{equation}
The term
\[
(\delta\lambda)(g,h)\coloneqq
\frac{\lambda(g) (g \cdot \lambda(h))}{\lambda(gh)}
\]
is a \emph{2-coboundary}. The set of all such coboundaries forms a subgroup $B^2(G, M) \subset Z^2(G, M)$. The quotient group is the second cohomology group of $G$ with coefficients in the $G$-module $M$:
\begin{equation}
H^2(G, M) \coloneqq Z^2(G, M) / B^2(G, M).
\end{equation}

The extension $E$ splits if and only if its class $[\beta]$ vanishes
in $H^2(G,M)$.


\subsection{The Schur multiplier of a finite abelian group}

We apply the extension theory of
Subsection~\ref{subsec:extensions-factor-sets} to central extensions of a finite
abelian group $V$ by the circle group $U(1)$.  For a finite group with
trivial action on the coefficients, there are isomorphisms
\[
H^2(V,U(1))\cong H^2(V,\C^\times)
\cong \Hom(H_2(V,\Z),U(1)).
\]
Thus this cohomology group is the Pontryagin dual of the Schur
multiplier $H_2(V,\Z)$; it is also sometimes called the
\emph{cohomological Schur multiplier}.  In the abelian setting, its
classes can be represented concretely by bicharacters.

Let $\Bil(V)$ denote the group of \emph{bicharacters}, defined as functions $B\colon V \times V \to U(1)$ that are homomorphisms in each argument. Every bicharacter satisfies the 2-cocycle identity, so $\Bil(V)$ is a subgroup of $Z^2(V, U(1))$.

We distinguish two subgroups within $\Bil(V)$:
\begin{align}
\Sym(V) &\coloneqq \{ B \in \Bil(V) \mid B(u, v) = B(v, u) \}, \\
\Alt(V) &\coloneqq \{ B \in \Bil(V) \mid B(u, u) = 1 \}.
\end{align}
Note that the alternating condition $B(u, u) = 1$ implies skew-symmetry, $B(u, v) = B(v, u)^{-1}$.

The \emph{antisymmetrization map}
$\mathcal{A}\colon Z^2(V,U(1))\to\Alt(V)$ is defined by
\begin{equation}
\mathcal{A}(\beta)(u,v)
  \coloneqq \beta(u,v)\beta(v,u)^{-1}.
\end{equation}
For the extension determined by $\beta$, the bicharacter
$\mathcal{A}(\beta)$ is its commutator bicharacter.

Restricting $\mathcal{A}$ to bicharacters yields a sequence
\begin{equation}\label{eq:bil-sym-alt}
1 \longrightarrow \Sym(V) \longrightarrow \Bil(V) \xrightarrow{\;\mathcal{A}\;} \Alt(V) \longrightarrow 1,
\end{equation}
whose kernel is precisely the subgroup of symmetric bicharacters. 

The exactness of sequence \eqref{eq:bil-sym-alt} (see \cite[\S~2]{Tambara2000}), together with the correspondence between central $U(1)$-extensions and $H^2(V,U(1))$ described in Subsection~\ref{subsec:extensions-factor-sets}, gives the group isomorphisms
\begin{equation}\label{eq:tambara-iso}
\frac{\Bil(V)}{\Sym(V)} \cong H^2(V, U(1)) \xrightarrow{\sim} \Alt(V).
\end{equation}
Every cohomology class has a bicharacter representative $\beta$, and
\eqref{eq:tambara-iso} sends $[\beta]$ to $\mathcal A(\beta)$. Two
bicharacter representatives determine the same class when their
quotient is symmetric.

\subsection{The Heisenberg extension}\label{subsec:heisenberg-extension}

Let $A$ be a finite abelian group, and let $\widehat{A} = \Hom(A, U(1))$ denote its Pontryagin dual. The \emph{double of $A$} is the direct sum
\[
V_A \coloneqq A \oplus \widehat{A}.
\]

We consider the Hilbert space $\HH = \C[A]$ with orthonormal basis $\{\ket{a} \mid a \in A\}$. For $a \in A$ and $\chi \in \widehat{A}$, define the \emph{shift operators} $X_a$ and \emph{phase operators} $Z_\chi$ by
\begin{equation}
X_a \ket{b} \coloneqq \ket{a+b}, \qquad Z_\chi \ket{b} \coloneqq \chi(b)\ket{b}.
\end{equation}
These operators satisfy the commutation relation $Z_\chi X_a = \chi(a) X_a Z_\chi$.
For elements $u = (a_u, \chi_u)$ and $v = (a_v, \chi_v)$ in $V_A$, we define the \emph{Weyl operator} $W_u \coloneqq X_{a_u} Z_{\chi_u}$.

The \emph{Heisenberg group} $H(A)$ is the subgroup of unitary operators generated by these Weyl operators together with the scalar subgroup $U(1)\cdot I$:
\begin{equation}
H(A) \coloneqq \{ z W_u \mid z \in U(1), u \in V_A \}.
\end{equation}

Using $Z_{\chi_u}X_{a_v}=\chi_u(a_v)X_{a_v}Z_{\chi_u}$, we obtain
\begin{equation}
W_u W_v = \chi_u(a_v) W_{u+v} = \beta_A(u,v) W_{u+v},
\end{equation}
where the $2$-cocycle $\beta_A \in \Bil(V_A)$ is given by
\begin{equation}
\beta_A((a, \chi), (b, \psi)) \coloneqq \chi(b).
\end{equation}
Thus $H(A)$ is a central extension of $V_A$ by $U(1)$. By
\eqref{eq:tambara-iso}, the class $[\beta_A]$ is determined by its
antisymmetrization, namely the \emph{symplectic form}
\begin{equation}
\omega_A(u, v) \coloneqq \mathcal{A}(\beta_A)(u, v) = \chi_u(a_v)\chi_v(a_u)^{-1}.
\end{equation}
Consequently,
\begin{equation}\label{eq:weyl-commutation}
W_uW_v=\omega_A(u,v)\,W_vW_u.
\end{equation}
Since $\omega_A$ is non-degenerate, the center of $H(A)$ is exactly the scalar subgroup $U(1)\cdot I$, and $V_A$ becomes a symplectic abelian group.


\subsection{The pseudo-symplectic group \texorpdfstring{$\Ps(A)$}{Ps(A)} and the Clifford group}

The \emph{Clifford group} $C(A)$ is the projective normalizer of the Heisenberg group, obtained by quotienting the unitary normalizer $N_{U(\HH)}(H(A))$ by global phases. Identifying $U(1)$ with the scalar subgroup $U(1)\cdot I$, we set
\begin{equation}
C(A) \coloneqq \frac{N_{U(\HH)}(H(A))}{U(1) \cdot I} = \left\{ [U] \in \frac{U(\HH)}{U(1)} \;\middle|\; U H(A) U^\dagger = H(A) \right\}.
\end{equation}

Let $[U] \in C(A)$. Since conjugation by $U$ preserves $H(A)$ and fixes the center $U(1)\cdot I$ pointwise, it induces an automorphism of the quotient
\[
H(A)/(U(1)\cdot I) \cong V_A.
\]
Hence there exists a group automorphism $T\colon V_A \to V_A$ such that the Weyl operators transform as
\begin{equation}\label{eq:conjugation-action}
    U W_u U^\dagger = \lambda(u) W_{Tu}, \quad \forall u \in V_A,
\end{equation}
for some function $\lambda\colon V_A \to U(1)$.

Because conjugation preserves the multiplication law in $H(A)$, it
preserves the Weyl commutation relation~\eqref{eq:weyl-commutation}.
Applying conjugation by $U$ to~\eqref{eq:weyl-commutation} gives
\[
U W_u W_v U^\dagger = \omega_A(u,v)\,U W_v W_u U^\dagger.
\]
Using \eqref{eq:conjugation-action}, we obtain
\[
\lambda(u)\lambda(v)\,W_{Tu}W_{Tv}
=
\omega_A(u,v)\,\lambda(v)\lambda(u)\,W_{Tv}W_{Tu},
\]
and therefore
\[
W_{Tu}W_{Tv}=\omega_A(u,v)\,W_{Tv}W_{Tu}.
\]
Since also $W_{Tu}W_{Tv}=\omega_A(Tu,Tv)\,W_{Tv}W_{Tu}$, we conclude that $T$ must preserve the symplectic form:
\begin{equation}
    \omega_A(Tu, Tv) = \omega_A(u, v).
\end{equation}
Thus $T \in \Sp(V_A)$. Moreover, comparing the coefficients of $W_{T(u+v)}$ in
\[
U(W_uW_v)U^\dagger = (UW_uU^\dagger)(UW_vU^\dagger),
\]
and using $W_{Tu}W_{Tv} = \beta_A(Tu, Tv) W_{T(u+v)}$, we obtain
\begin{equation}\label{eq:coboundary-condition}
    \frac{\lambda(u+v)}{\lambda(u)\lambda(v)} = \frac{\beta_A(Tu, Tv)}{\beta_A(u, v)}.
\end{equation}

We encode this data in the \emph{pseudo-symplectic group} $\Ps(A)$,
using the model of Weil~\cite{Weil1964}; see also~\cite{Galindo2017}.

\begin{definition}
The group $\Ps(A)$ consists of the pairs
$(T,\lambda)\in\Sp(V_A)\times\Map(V_A,U(1))$ satisfying
\eqref{eq:coboundary-condition} for every $u,v\in V_A$.
The group operation is the twisted product (cf.~\cite[Eq.~7]{Weil1964}):
\begin{equation}\label{eq:twisted-product}
    (T, \lambda) \cdot (S, \mu) = (TS, \lambda^S \cdot \mu),
\end{equation}
where $(\lambda^S \cdot \mu)(u) \coloneqq \lambda(Su)\mu(u)$, and the identity element is $(\id,\mathbf{1})$, with $\mathbf{1}$ the constant function $1$.
The defining condition is preserved by this product, and
\[
(T,\lambda)^{-1}
=\left(T^{-1},u\longmapsto\lambda(T^{-1}u)^{-1}\right),
\]
so these formulas indeed define a group.
\end{definition}

\begin{theorem}\label{thm:clifford-structure}
The map $\Phi\colon \Ps(A) \to C(A)$, assigning to $(T, \lambda)$ the unique projective class of unitary operators implementing the action $U W_u U^\dagger = \lambda(u) W_{Tu}$, is a group isomorphism.
\end{theorem}

\begin{proof}
\emph{Well-definedness:} Let
\[
\rho\colon H(A)\longrightarrow U(\HH),
\qquad \rho(zW_u)=zW_u,
\]
denote the Weyl representation defined in
Subsection~\ref{subsec:heisenberg-extension}.  For
$(T,\lambda)\in\Ps(A)$, we construct a representation of $H(A)$ with
the same central character as $\rho$ and then apply the
Stone--von Neumann--Mackey theorem.

Define $\rho'\colon H(A) \to U(\HH)$ by
\[
    \rho'(z W_u) \coloneqq z \lambda(u) W_{Tu}.
\]
We verify that $\rho'$ is a representation. Using the Weyl multiplication law, we compute
\begin{align*}
    \rho'(W_u)\rho'(W_v) &= \lambda(u)\lambda(v) W_{Tu} W_{Tv} \\
    &= \lambda(u)\lambda(v) \beta_A(Tu, Tv) W_{T(u+v)}.
\end{align*}
Using \eqref{eq:coboundary-condition}, this becomes
\[
    \rho'(W_u)\rho'(W_v) = \beta_A(u, v) \rho'(W_{u+v}) = \rho'(\beta_A(u, v) W_{u+v}) = \rho'(W_u W_v).
\]
Hence $\rho'$ is a unitary representation of $H(A)$ on $\HH$.
The Weyl representation $\rho$ is irreducible.  Indeed, an operator
commuting with every $Z_\chi$ is diagonal in the basis
$\{\ket{a}\mid a\in A\}$ because the characters of $A$ separate points;
if it also commutes with every $X_a$, all its diagonal entries are
equal.  Thus the commutant of the Weyl operators consists only of
scalars.  Since $T$ is an automorphism and the values of $\lambda$ are
nonzero, the family $\{\rho'(W_u)\mid u\in V_A\}$ is, up to nonzero
scalar factors, the family $\{\rho(W_w)\mid w\in V_A\}$ reindexed by
$w=Tu$.  Hence $\rho'$ is irreducible
as well. Evaluating \eqref{eq:coboundary-condition} at $u=v=0$ gives $\lambda(0)=1$, and therefore
\[
\rho'(zI) = z \lambda(0) W_{T0} = zI.
\]
Thus $\rho'$ and $\rho$ both restrict to $z\mapsto zI$ on the center.

By the Stone--von Neumann--Mackey
theorem~\cite{Mackey1949,Prasad2011}, there exists a unitary $U$,
unique up to a scalar, such that
\[
U \rho(h) U^\dagger = \rho'(h), \qquad h \in H(A).
\]
Restricting to Weyl operators yields
\[
U W_u U^\dagger = \lambda(u) W_{Tu}.
\]

\emph{Homomorphism property:} Let $U$ and $V$ be unitary
representatives of $\Phi(T,\lambda)$ and $\Phi(S,\mu)$, respectively.
Then
\begin{align*}
(UV)W_u(UV)^\dagger
  &=U(VW_uV^\dagger)U^\dagger \\
  &=U\bigl(\mu(u)W_{Su}\bigr)U^\dagger \\
  &=\mu(u)\lambda(Su)W_{TSu}.
\end{align*}
Hence $UV$ implements the pair $(TS,\lambda^S\cdot\mu)$, which is
$(T,\lambda)(S,\mu)$ by \eqref{eq:twisted-product}.  Thus $\Phi$ is a
homomorphism.
For surjectivity, \eqref{eq:conjugation-action} and
\eqref{eq:coboundary-condition} show that every class in $C(A)$
determines a pair $(T,\lambda)\in\Ps(A)$.  For
injectivity, suppose that
$\Phi(T,\lambda)=\Phi(S,\mu)$.  Representatives of these two equal
projective classes differ by a scalar and therefore induce the same
conjugation action on every Weyl operator.  Passing to
$H(A)/(U(1)\cdot I)$ gives $Tu=Su$ for every $u\in V_A$, so $T=S$.
The equality
$\lambda(u)W_{Tu}=\mu(u)W_{Tu}$ then gives
$\lambda(u)=\mu(u)$ for every $u$, and hence $(T,\lambda)=(S,\mu)$.
\end{proof}
\subsection{The exact sequence and the obstruction cocycle}

By Theorem~\ref{thm:clifford-structure}, we identify $C(A)$ with
$\Ps(A)$ and write its elements as pairs $(T,\lambda)$.  Define
\[
\pi\colon C(A) \to \Sp(V_A), \qquad \pi(T,\lambda)=T.
\]
The projection $\pi$ is surjective. Indeed, for $T\in\Sp(V_A)$, the bicharacter
\[
\gamma_T(u,v)\coloneqq\frac{\beta_A(Tu,Tv)}{\beta_A(u,v)}
\]
has trivial antisymmetrization because $T$ preserves $\omega_A=\mathcal A(\beta_A)$. By \eqref{eq:tambara-iso}, $\gamma_T$ is therefore cohomologically trivial.  Thus there is a normalized cochain
$f_T\colon V_A\to U(1)$ such that
\[
\gamma_T(u,v)=\frac{f_T(u)f_T(v)}{f_T(u+v)}.
\]
Setting $\lambda_T(u)\coloneqq f_T(u)^{-1}$ gives
\[
\frac{\lambda_T(u+v)}{\lambda_T(u)\lambda_T(v)}
=\gamma_T(u,v)
=\frac{\beta_A(Tu,Tv)}{\beta_A(u,v)},
\]
so $(T,\lambda_T)\in\Ps(A)$.

The kernel $K=\ker(\pi)$ consists of pairs $(\id,\mu)$ with trivial
symplectic part. Substituting $T=\id$ into
\eqref{eq:coboundary-condition}, we obtain
\[
\mu(u+v)=\mu(u)\mu(v),
\]
hence $(\id,\mu)\mapsto\mu$ is an isomorphism from $K$ to the
Pontryagin dual $\widehat{V_A}$. Since $\omega_A$ is non-degenerate, it induces an isomorphism
\[
\kappa_A\colon V_A \xrightarrow{\sim} \widehat{V_A},
\]
defined by
\begin{equation}
\kappa_A(v)(u) \coloneqq \omega_A(v, u).
\end{equation}
Using $\kappa_A$, we identify $K$ with $V_A$. The inclusion of the kernel is then
\[
\nu\colon V_A \to C(A), \qquad \nu(v)=(\id,\kappa_A(v)).
\]
These identifications give the short exact sequence
\begin{equation}\label{eq:clifford-ses-final}
1 \longrightarrow V_A \xrightarrow{\;\nu\;} C(A) \xrightarrow{\;\pi\;} \Sp(V_A) \longrightarrow 1.
\end{equation}

The extension~\eqref{eq:clifford-ses-final} induces an action of
$\Sp(V_A)$ on $V_A$ by conjugation. If $v \in V_A$ and
$g=(T,\lambda)\in C(A)$ is any lift of $T$, then
\begin{equation}
(T,\lambda)\,\nu(v)\,(T,\lambda)^{-1}=\nu(Tv).
\end{equation}
Hence the induced $\Sp(V_A)$-module structure is $T\cdot v=Tv$.

We now construct a $2$-cocycle representing
\eqref{eq:clifford-ses-final}. Choose a normalized set-theoretic section
\[
s\colon\Sp(V_A)\to C(A), \qquad s(T)=(T,\lambda_T),
\]
where each $\lambda_T\colon V_A\to U(1)$ satisfies
\eqref{eq:coboundary-condition}, and where
$\lambda_{\id}=\mathbf 1$.  For $T,S\in\Sp(V_A)$ define
\begin{equation}\label{eq:right-defect}
r_{T,S}(u)
\coloneqq
\frac{\lambda_T(Su)\lambda_S(u)}{\lambda_{TS}(u)}.
\end{equation}
The numerator and denominator in \eqref{eq:right-defect} are phase
functions lifting the same symplectic map $TS$, so their ratio is a
linear character of $V_A$.  Let
\[
d_s(T,S)\coloneqq\kappa_A^{-1}(r_{T,S})\in V_A.
\]
The twisted product gives a factorization on the right,
\[
s(T)s(S)=s(TS)\nu(d_s(T,S)).
\]
To compare this with the left-factor convention of
Subsection~\ref{subsec:extensions-factor-sets}, observe that
$s(TS)\nu(v)=\nu(TS v)s(TS)$.  Thus the factor set in this convention is
\begin{equation}\label{eq:obstruction-factor}
\cO_s(T,S)\coloneqq TS\,d_s(T,S),
\qquad
s(T)s(S)=\nu(\cO_s(T,S))s(TS).
\end{equation}
The map
\[
\cO_s\colon\Sp(V_A)\times\Sp(V_A)\to V_A
\]
is the normalized \emph{obstruction $2$-cocycle} associated with the
section $s$ for the left action $T\cdot v=Tv$ of $\Sp(V_A)$ on $V_A$.
Comparing the two parenthesizations $(s(T)s(S))s(R)$ and
$s(T)(s(S)s(R))$ using \eqref{eq:obstruction-factor} gives, in the
additive notation of $V_A$, the identity
\[
T\cO_s(S,R)+\cO_s(T,SR)
=\cO_s(T,S)+\cO_s(TS,R).
\]
Its cohomology class
\[
[\cO_s]\in H^2(\Sp(V_A),V_A)
\]
is independent of the choice of section. We denote this intrinsic extension class by
\[
[\cO_A]\coloneqq [\cO_s]\in H^2(\Sp(V_A),V_A).
\]
It is the obstruction to splitting the exact sequence \eqref{eq:clifford-ses-final}.

\section{Decomposition and splitting reductions}\label{sec:decomposition_reduction}

The obstruction class decomposes along coprime factors and vanishes for
groups of odd order.  We also prove that splitting descends to direct
summands.

\subsection{Coprime decomposition}

Let
\[
A = B_1 \oplus B_2,
\qquad \gcd(|B_1|,|B_2|)=1,
\]
where $B_1$ and $B_2$ are finite abelian groups. Using
$V_G=G\oplus\widehat G$ and
$\widehat{B_1\oplus B_2}\cong\widehat{B_1}\oplus\widehat{B_2}$, we obtain
\[
V_A = V_{B_1} \oplus V_{B_2},
\]
and the symplectic form splits as the orthogonal sum
\[
\omega_A=\omega_{B_1}\oplus\omega_{B_2}.
\]

The symplectic group respects this splitting. Let $T \in \Sp(V_A)$. Because automorphisms preserve element orders, the subsets of elements whose orders divide $|V_{B_1}|$ and $|V_{B_2}|$ are characteristic in $V_A$. Since $|V_{B_1}|$ and $|V_{B_2}|$ are coprime, these subsets are precisely the direct summands $V_{B_1}$ and $V_{B_2}$. Hence both summands are $T$-invariant, and every symplectic automorphism decomposes uniquely. We obtain the isomorphism
\begin{equation}\label{eq:coprime-symplectic-product}
\Sp(V_{B_1}) \times \Sp(V_{B_2})
\xrightarrow{\sim}\Sp(V_A),
\qquad (T_1,T_2)\longmapsto T_1\oplus T_2.
\end{equation}

Via Theorem~\ref{thm:clifford-structure}, we may describe elements of each Clifford group as pairs $(T_j,\lambda_j)$, for $j\in\{1,2\}$, satisfying the coboundary condition. We define
\begin{align}
\Psi\colon C(B_1) \times C(B_2) &\longrightarrow C(B_1 \oplus B_2) \nonumber \\
((T_1, \lambda_1), (T_2, \lambda_2)) &\longmapsto (T_1 \oplus T_2, \lambda_1 \boxtimes \lambda_2),
\end{align}
where
\[
(\lambda_1 \boxtimes \lambda_2)(u_1 + u_2) \coloneqq \lambda_1(u_1)\lambda_2(u_2),
\qquad u_j\in V_{B_j}\quad(j\in\{1,2\}).
\]
Likewise, the bicharacter on $V_A=V_{B_1}\oplus V_{B_2}$ is given by
\[
\beta_A(u_1+u_2,v_1+v_2)
=\beta_{B_1}(u_1,v_1)\beta_{B_2}(u_2,v_2).
\]
Substituting the formulas for $\lambda_1\boxtimes\lambda_2$ and
$\beta_A$ into \eqref{eq:coboundary-condition} shows that
$\lambda_1\boxtimes\lambda_2$ satisfies the required condition.
Under these identifications, the product \eqref{eq:twisted-product}
is componentwise; hence $\Psi$ is a group homomorphism.

\begin{proposition}\label{prop:coprime-decomp}
Let $A = B_1 \oplus B_2$ with $\gcd(|B_1|, |B_2|) = 1$. The map $\Psi$ induces an isomorphism of short exact sequences:
\begin{equation}
\begin{tikzcd}[column sep=small]
1 \arrow[r] & V_{B_1} \oplus V_{B_2} \arrow[r] \arrow[d, "\cong"] & C(B_1) \times C(B_2) \arrow[r] \arrow[d, "\Psi", "\cong"'] & \Sp(V_{B_1}) \times \Sp(V_{B_2}) \arrow[r] \arrow[d, "\cong"] & 1 \\
1 \arrow[r] & V_A \arrow[r] & C(A) \arrow[r] & \Sp(V_A) \arrow[r] & 1.
\end{tikzcd}
\end{equation}
Let $G_j=\Sp(V_{B_j})$ for $j\in\{1,2\}$, and regard $V_{B_1}$ and
$V_{B_2}$ as $G_1\times G_2$-modules through the projections
$p_j\colon G_1\times G_2\to G_j$.  Let
$\jmath_j\colon V_{B_j}\to V_{B_1}\oplus V_{B_2}$ be the equivariant
inclusions.  The notation $p_j^*$ denotes inflation and $(\jmath_j)_*$ denotes
the map induced by change of coefficients.  Then
the obstruction class decomposes precisely as
\begin{equation}
[\cO_A]
=(\jmath_1)_*p_1^*[\cO_{B_1}]
 +(\jmath_2)_*p_2^*[\cO_{B_2}]
 \in H^2(G_1\times G_2,V_{B_1}\oplus V_{B_2}).
\end{equation}
In particular, the extension for $A$ splits if and only if the
extensions for both $B_1$ and $B_2$ split.
\end{proposition}

\begin{proof}
The commutativity of the diagram follows from the definition of $\Psi$. On the kernel, $\Psi$ maps the pair of characters $(\kappa_{B_1}(v_1),\kappa_{B_2}(v_2))$ to their external product.  Explicitly, for $u_j,v_j\in V_{B_j}$ with $j\in\{1,2\}$,
\[
(\kappa_{B_1}(v_1)\boxtimes\kappa_{B_2}(v_2))(u_1+u_2)
=\kappa_{B_1}(v_1)(u_1)\kappa_{B_2}(v_2)(u_2)
=\kappa_A(v_1+v_2)(u_1+u_2),
\]
which is the identification $V_{B_1}\oplus V_{B_2}\cong V_A$.
On the quotient, $\Psi$ induces the symplectic-group isomorphism
\eqref{eq:coprime-symplectic-product}. The short five lemma for groups
therefore implies that $\Psi$ is an isomorphism.  If $s_1$ and $s_2$
are normalized sections, the product section has factor set, for
$g_j,h_j\in G_j$ with $j\in\{1,2\}$,
\[
\bigl((g_1,g_2),(h_1,h_2)\bigr)
\longmapsto
\bigl(\cO_{s_1}(g_1,h_1),\cO_{s_2}(g_2,h_2)\bigr),
\]
which proves the asserted formula for the extension class.
To prove the final assertion, splittings of the two factors combine
through $\Psi$ to produce a splitting for $A$. Conversely, transport a
splitting for $A$ through $\Psi^{-1}$, restrict it to
$G_1\times\{1\}$ and $\{1\}\times G_2$, and project to the corresponding
middle factors. The resulting homomorphisms are splittings for $B_1$
and $B_2$.
\end{proof}

We apply Proposition~\ref{prop:coprime-decomp} to the primary
decomposition of $A$. The group decomposes uniquely as
\begin{equation}
A = A_{\mathrm{odd}} \oplus A_2,
\end{equation}
where $A_{\mathrm{odd}}$ is the subgroup of all elements of odd order and $A_2$ is the $2$-primary component of $A$. Since $|A_2|$ and $|A_{\mathrm{odd}}|$ are coprime, Proposition~\ref{prop:coprime-decomp} implies that the splitting problem for $C(A)$ reduces to the independent analysis of $C(A_{\mathrm{odd}})$ and $C(A_2)$.

\subsection{Splitting at odd primes}

\begin{proposition}\label{prop:odd_case_vanish}
If $|A|$ is odd, the obstruction class $[\cO_A]$ vanishes. The Clifford group splits as a semidirect product $C(A) \cong V_A \rtimes \Sp(V_A)$, with the splitting given explicitly by the section $T\mapsto(T,\lambda_T)$, where $\lambda_T$ is defined in~\eqref{eq:odd-section}.
\end{proposition}

\begin{proof}
Let $m$ be the exponent of $V_A$, which is odd, and put
\[
\mu_m\coloneqq\{\zeta\in U(1)\mid \zeta^m=1\}.
\]
Every bicharacter on $V_A$ takes values in $\mu_m$.  Choose an integer
$r$ such that $2r\equiv1\pmod m$ and define
$\zeta^{1/2}\coloneqq\zeta^r$ for $\zeta\in\mu_m$.  This square-root
operation is the inverse of squaring on $\mu_m$ and is a group
homomorphism.

Define the function
\[
q_A(u)\coloneqq\beta_A(u,u)^{1/2}.
\]
For $T\in\Sp(V_A)$ set
\begin{equation}\label{eq:odd-section}
\lambda_T(u)
\coloneqq\frac{q_A(Tu)}{q_A(u)}
=\left(\frac{\beta_A(Tu,Tu)}{\beta_A(u,u)}\right)^{1/2}.
\end{equation}

Since $T$ preserves $\omega_A=\mathcal A(\beta_A)$, for all
$u,v\in V_A$ we have
\[
\frac{\beta_A(Tu,Tv)}{\beta_A(u,v)}
=
\frac{\beta_A(Tv,Tu)}{\beta_A(v,u)}.
\]
Bilinearity of $\beta_A$ and the identity
$q_A(x)^2=\beta_A(x,x)$ give
\begin{align*}
\left(
\frac{\lambda_T(u+v)}{\lambda_T(u)\lambda_T(v)}
\right)^2
&=
\frac{\beta_A(T(u+v),T(u+v))}
     {\beta_A(Tu,Tu)\beta_A(Tv,Tv)}
\frac{\beta_A(u,u)\beta_A(v,v)}
     {\beta_A(u+v,u+v)} \\
&=
\frac{\beta_A(Tu,Tv)\beta_A(Tv,Tu)}
     {\beta_A(u,v)\beta_A(v,u)} \\
&=
\left(
\frac{\beta_A(Tu,Tv)}{\beta_A(u,v)}
\right)^2.
\end{align*}
Both fractions inside the squares belong to $\mu_m$.  Since squaring
is injective on $\mu_m$, it follows that
\[
\frac{\lambda_T(u+v)}{\lambda_T(u)\lambda_T(v)}
=\frac{\beta_A(Tu,Tv)}{\beta_A(u,v)}.
\]
Hence $(T,\lambda_T)\in\Ps(A)$.  Moreover, for all $T,S$,
\[
\lambda_T(Su)\lambda_S(u)
=\frac{q_A(TSu)}{q_A(Su)}\frac{q_A(Su)}{q_A(u)}
=\lambda_{TS}(u).
\]
Therefore $T\mapsto(T,\lambda_T)$ is a homomorphic section, proving
the proposition.
\end{proof}

Combining Propositions~\ref{prop:odd_case_vanish}
and~\ref{prop:coprime-decomp}, we conclude that for any finite abelian
group, the splitting obstruction $[\cO_A]$ is determined entirely by
the $2$-primary component.
\subsection{Reduction to direct summands}

We now show that splitting descends to direct summands. Let
\[
A = B \oplus C
\]
be an arbitrary decomposition of finite abelian groups; no coprimality assumption is imposed. Then $V_A$ decomposes orthogonally as
\[
V_A = V_B \oplus V_C.
\]
Define the embeddings
\[
i\colon V_B\hookrightarrow V_A,\qquad i(v_B)=(v_B,0_{V_C}),
\]
and
\[
j\colon\Sp(V_B)\hookrightarrow \Sp(V_A),\qquad j(S)=S\oplus \id_{V_C}.
\]

Using Theorem~\ref{thm:clifford-structure}, we define a map lifting
$j$.  For $(S,\mu)\in \Ps(B)$, define
\[
\widetilde\mu\colon V_A\to U(1),
\qquad
\widetilde\mu(v_B,v_C)\coloneqq\mu(v_B),
\]
and set
\begin{equation}
\iota(S, \mu) \coloneqq (S \oplus \id_{V_C}, \widetilde\mu).
\end{equation}
Since
\[
\beta_A((u_B,u_C),(v_B,v_C))
=\beta_B(u_B,v_B)\beta_C(u_C,v_C),
\]
the function $\widetilde\mu$ satisfies the coboundary condition for
$\beta_A$; the $C$-component is trivial because the symplectic map
acts as the identity on $V_C$. Moreover, the twisted product is
preserved, and therefore $\iota$ is a group homomorphism.
It is injective, and on the kernel it satisfies
\[
\iota(\nu_B(v_B))=\nu_A(i(v_B)),
\]
because
$\omega_A((v_B,0_{V_C}),(u_B,u_C))=\omega_B(v_B,u_B)$.

The maps $i$, $j$, and $\iota$ yield a commutative diagram of exact
sequences (using the identification $C(A)=\Ps(A)$):
\begin{equation}
\begin{tikzcd}
1 \arrow[r] & V_B \arrow[r] \arrow[d, "i"] & C(B) \arrow[r] \arrow[d, "\iota"] & \Sp(V_B) \arrow[r] \arrow[d, "j"] & 1 \\
1 \arrow[r] & V_A \arrow[r] & C(A) \arrow[r] & \Sp(V_A) \arrow[r] & 1.
\end{tikzcd}
\end{equation}

\medskip

\begin{theorem}\label{thm:embedding}
Let $A = B \oplus C$ be a decomposition of finite abelian groups. If the Clifford extension for $A$ splits, then the Clifford extension for $B$ splits.
\end{theorem}

\begin{proof}
Assume that the extension for $A$ splits. Then there exists a normalized group homomorphism
\[
\sigma_A\colon \Sp(V_A) \to C(A)
\]
satisfying $\pi_A \circ \sigma_A = \id_{\Sp(V_A)}$.  We use
$\sigma_A$ to construct a splitting section for the Clifford extension
of $B$.

For $S \in \Sp(V_B)$, set
\[
\Sigma_S = \sigma_A(j(S)) \in C(A).
\]
Via the identification $C(A)=\Ps(A)$ from Theorem~\ref{thm:clifford-structure}, we may write
\begin{equation}
\Sigma_S = (S \oplus \id_{V_C}, \Lambda_S),
\end{equation}
where $\Lambda_S\colon V_A \to U(1)$ satisfies the coboundary condition
\begin{equation}
\frac{\Lambda_S(u+v)}{\Lambda_S(u)\Lambda_S(v)} = \frac{\beta_A((S\oplus \id_{V_C})u, (S\oplus \id_{V_C})v)}{\beta_A(u, v)}.
\end{equation}

We now restrict $\Lambda_S$ to the copy of $V_B$ inside $V_A$. Define
\begin{equation}
\lambda_S(u_B) \coloneqq \Lambda_S(u_B, 0_{V_C}).
\end{equation}
We verify that $\lambda_S$ satisfies the coboundary condition for $\beta_B$. For $u=(u_B,0_{V_C})$ and $v=(v_B,0_{V_C})$ in $V_B\oplus \{0_{V_C}\}$, the factorization of $\beta_A$ gives
\begin{equation}
\beta_A(u, v) = \beta_B(u_B, v_B) \cdot \beta_C(0_{V_C}, 0_{V_C}) = \beta_B(u_B, v_B),
\end{equation}
since $\beta_C(0_{V_C},0_{V_C})=1$. Moreover,
$(S \oplus \id_{V_C})(u_B, 0_{V_C}) = (Su_B, 0_{V_C})$, and therefore
\begin{equation}
\frac{\lambda_S(u_B + v_B)}{\lambda_S(u_B)\lambda_S(v_B)} = \frac{\beta_B(Su_B, Sv_B)}{\beta_B(u_B, v_B)}.
\end{equation}
Hence $(S, \lambda_S)$ is an element of $\Ps(B)$.

Define
\[
\sigma_B\colon\Sp(V_B)\to \Ps(B),\qquad \sigma_B(S)\coloneqq (S,\lambda_S).
\]
It remains to prove that $\sigma_B$ is a group homomorphism.

Let $S_1, S_2 \in \Sp(V_B)$. Since $\sigma_A$ is a homomorphism,
\begin{equation}
\Sigma_{S_1 S_2} = \Sigma_{S_1} \cdot \Sigma_{S_2}.
\end{equation}
Using the twisted product in $\Ps(A)$, we obtain
\begin{equation}
\Sigma_{S_1} \cdot \Sigma_{S_2} = ((S_1 S_2) \oplus \id_{V_C}, \, \Lambda_{S_1}^{S_2 \oplus \id_{V_C}} \cdot \Lambda_{S_2}),
\end{equation}
where $(\Lambda_{S_1}^{S_2 \oplus \id_{V_C}} \cdot \Lambda_{S_2})(u) = \Lambda_{S_1}((S_2 \oplus \id_{V_C})u) \cdot \Lambda_{S_2}(u)$.

Comparing with $\Sigma_{S_1 S_2} = ((S_1 S_2) \oplus \id_{V_C}, \Lambda_{S_1 S_2})$, we obtain
\begin{equation}
\Lambda_{S_1 S_2}(u) = \Lambda_{S_1}((S_2 \oplus \id_{V_C})u) \cdot \Lambda_{S_2}(u).
\end{equation}

Restricting to $u = (u_B, 0_{V_C}) \in V_B \oplus \{0_{V_C}\}$ gives
\begin{align}
\lambda_{S_1 S_2}(u_B) &= \Lambda_{S_1 S_2}(u_B, 0_{V_C}) \nonumber \\
&= \Lambda_{S_1}((S_2 \oplus \id_{V_C})(u_B, 0_{V_C})) \cdot \Lambda_{S_2}(u_B, 0_{V_C}) \nonumber \\
&= \Lambda_{S_1}(S_2 u_B, 0_{V_C}) \cdot \Lambda_{S_2}(u_B, 0_{V_C}) \nonumber \\
&= \lambda_{S_1}(S_2 u_B) \cdot \lambda_{S_2}(u_B).
\end{align}
The resulting identity is the twisted product rule in $\Ps(B)$:
\begin{equation}
(S_1, \lambda_{S_1}) \cdot (S_2, \lambda_{S_2}) = (S_1 S_2, \lambda_{S_1}^{S_2} \cdot \lambda_{S_2}) = (S_1 S_2, \lambda_{S_1 S_2}).
\end{equation}
Hence $\sigma_B(S_1 S_2) = \sigma_B(S_1) \cdot \sigma_B(S_2)$, and therefore $\sigma_B$ is a group homomorphism.

Finally, $\sigma_B$ is normalized because $\sigma_B(\id) = (\id, \lambda_{\id})$ and $\lambda_{\id}(u_B) = \Lambda_{\id}(u_B, 0_{V_C}) = 1$, using that $\sigma_A$ is normalized. Since $\pi_B(\sigma_B(S)) = S$ by construction, $\sigma_B$ is a splitting section for the Clifford extension for $B$.
\end{proof}

\section{Nonsplitting for \texorpdfstring{$\Z_{2^k}$}{Z/(2\string^k)}}\label{sec:nonsplit-cyclic}

For $A=\Z_N$, Korbel\'a\v{r} and Tolar proved that, when $N$ is even,
the Clifford extension splits if and only if $4\nmid N$~\cite{KT23}.
Their proof uses the explicit matrix model of the cyclic Clifford group
and its symplectic action~\cite{Appleby2005}.  For $N=2^k$ with
$k\geq2$, we prove nonsplitting in the pseudo-symplectic model of
Theorem~\ref{thm:clifford-structure}, using the relations $t^N=I$ and
$(st)^3=s^2$ for the matrices defined in
\eqref{eq:cyclic-generators}.

For $p\in\Z_N$, let
\[
\chi_p(b)\coloneqq \exp\!\left(\frac{2\pi \mathrm{i}}{N}pb\right),
\qquad b\in\Z_N.
\]
We identify $p$ with $\chi_p\in\widehat{\Z_N}$ and thereby identify
$\widehat{\Z_N}\cong \Z_N$.  We then write
\[
V_{\Z_N} \coloneqq \Z_N \oplus \Z_N,
\]
with elements $u = (a, p)$ and $v = (b, q)$. Throughout
Section~\ref{sec:nonsplit-cyclic}, every coordinate occurring in an
exponential or parity expression denotes its unique integer
representative in $\{0,\dots,N-1\}\subset \Z$. The Heisenberg 2-cocycle is
\begin{equation}
\beta_{\Z_N}(u, v) = \exp\!\left(\frac{2\pi \mathrm{i}}{N}\, pb\right),
\end{equation}
and we use Theorem~\ref{thm:clifford-structure} to identify $C(\Z_N)$ with pairs $(T, \lambda)$ satisfying the coboundary condition
\begin{equation}
\frac{\lambda(u + v)}{\lambda(u)\lambda(v)} = \frac{\beta_{\Z_N}(Tu, Tv)}{\beta_{\Z_N}(u, v)}.
\end{equation}
Under this identification,
\[
\omega_{\Z_N}((a,p),(b,q))
=\exp\!\left(\frac{2\pi \mathrm{i}}{N}(pb-qa)\right),
\]
and therefore
$\Sp(V_{\Z_N})=\Sp(2,\Z_N)=\SL(2,\Z_N)$.

We use the generators
\begin{equation}\label{eq:cyclic-generators}
t = \begin{pmatrix} 1 & 1 \\ 0 & 1 \end{pmatrix}, \qquad
s = \begin{pmatrix} 0 & -1 \\ 1 & 0 \end{pmatrix}.
\end{equation}
\begin{samepage}
Korbel\'a\v{r} and Tolar use $t$ together with the conjugate
lower-unipotent generator
\[
r=\begin{pmatrix}1&0\\-1&1\end{pmatrix}=sts^{-1}
\]
in their presentation of $\SL(2,\Z_N)$~\cite{KT23}.
\end{samepage}
The matrices $t$ and $s$ act on $V_{\Z_N}$ by
$t(a,p)=(a+p,p)$ and $s(a,p)=(-p,a)$. A splitting homomorphism
$\sigma\colon \SL(2, \Z_N) \to C(\Z_N)$ is determined by its values
$\tilde{t}=\sigma(t)$ and $\tilde{s}=\sigma(s)$.
\subsection{Explicit lifts of the generators}

For the generator $t$, we compute the ratio of cocycles directly using the action $t(a, p) = (a+p, p)$. For any $u=(a,p)$ and $v=(b,q)$,
\begin{equation}
\frac{\beta_{\Z_N}(tu, tv)}{\beta_{\Z_N}(u, v)} = \frac{\exp(\frac{2\pi \mathrm{i}}{N} p(b+q))}{\exp(\frac{2\pi \mathrm{i}}{N} pb)} = \exp\!\left(\frac{2\pi \mathrm{i}}{N}\, pq\right).
\end{equation}
A particular solution is $\lambda_t^{(0)}(a, p) = \exp(\pi \mathrm{i} p^2/N)$, since
\begin{equation}
\frac{\lambda_t^{(0)}(u + v)}{\lambda_t^{(0)}(u)\lambda_t^{(0)}(v)} = \exp\!\left(\frac{\pi \mathrm{i}}{N}\bigl((p + q)^2 - p^2 - q^2\bigr)\right) = \exp\!\left(\frac{2\pi \mathrm{i}}{N}\, pq\right).
\end{equation}
This function is well defined on $\Z_N$ because $N$ is even: replacing
$p$ by $p+N$ changes its value by
$\exp(2\pi \mathrm{i} p+\pi \mathrm{i}N)=1$.  The same observation justifies the use of
unreduced integer representatives in the displayed calculation.
Any other solution differs from $\lambda_t^{(0)}$ by a linear character of $V_{\Z_N}$. Parametrizing these characters as $\eta_{(x,y)}(a, p) = \exp(2\pi \mathrm{i}(xa + yp)/N)$, the general lift of $t$ takes the form
\begin{equation}\label{eq:lift-t}
\tilde{t} = (t, \lambda_t), \qquad \lambda_t(a, p) = \exp\!\left(\frac{\pi \mathrm{i}}{N} p^2 + \frac{2\pi \mathrm{i}}{N}(xa + yp)\right),
\end{equation}
determined by a unique pair $(x, y) \in \Z_N^2$.

For the generator $s$, the coboundary ratio is
\begin{equation}
\frac{\beta_{\Z_N}(su, sv)}{\beta_{\Z_N}(u, v)} = \frac{\exp(\frac{2\pi \mathrm{i}}{N} a(-q))}{\exp(\frac{2\pi \mathrm{i}}{N} pb)} = \exp\!\left(-\frac{2\pi \mathrm{i}}{N} (aq + pb)\right).
\end{equation}
A particular solution is $\lambda_s^{(0)}(a, p) = \exp(-2\pi \mathrm{i}\, ap/N)$, since
\[
\frac{\lambda_s^{(0)}(u+v)}{\lambda_s^{(0)}(u)\lambda_s^{(0)}(v)} = \exp\!\left(-\frac{2\pi \mathrm{i}}{N}\bigl((a+b)(p+q) - ap - bq\bigr)\right) = \exp\!\left(-\frac{2\pi \mathrm{i}}{N}(aq+bp)\right).
\]
Hence every lift of $s$ is of the form
\begin{equation}\label{eq:lift-s}
\tilde{s} = (s, \lambda_s), \qquad \lambda_s(a, p) = \exp\!\left(-\frac{2\pi \mathrm{i}}{N} ap + \frac{2\pi \mathrm{i}}{N}(za + wp)\right),
\end{equation}
for a pair $(z, w) \in \Z_N^2$.

\subsection{Incompatible lifting constraints}

Assume that a splitting exists.  Then the lifts of $t$ and $s$ satisfy
the relations $t^N=I$ and $(st)^3=s^2$.  The next two lemmas determine
the resulting restrictions on the parameter $x$ in
\eqref{eq:lift-t}.

\begin{lemma}\label{lem:parity-constraint}
Let $N=2^k$ with $k \geq 1$. If $\tilde{t}^N = (\id, \mathbf{1})$, then the parameter $x$ in \eqref{eq:lift-t} must be odd.
\end{lemma}

\begin{proof}
Using the product law $(T, \lambda) \cdot (S, \mu) = (TS, \lambda^S \mu)$ iteratively, the phase of $\tilde{t}^N$ is given by
\begin{equation}
\Lambda_N(u) = \prod_{j=0}^{N-1} \lambda_t(t^j u).
\end{equation}
Since $t^j(a, p) = (a + jp, p)$, the second component $p$ remains constant. Substituting $\lambda_t$ from \eqref{eq:lift-t} gives
\begin{equation}
\Lambda_N(a, p) = \prod_{j=0}^{N-1} \exp\!\left(\frac{\pi \mathrm{i}}{N} p^2 + \frac{2\pi \mathrm{i}}{N}(x(a+jp) + yp)\right).
\end{equation}
We separate the product into quadratic and linear parts. The quadratic term is constant in $j$:
\[
\prod_{j=0}^{N-1} \exp\!\left(\frac{\pi \mathrm{i}}{N} p^2\right) = \exp(\pi \mathrm{i} p^2) = (-1)^{p^2} = (-1)^p,
\]
where we used that $p^2 \equiv p \pmod 2$ for integers. The linear term involves a sum over $j$:
\[
\sum_{j=0}^{N-1} \frac{2\pi \mathrm{i}}{N}(xa + xjp + yp) = 2\pi \mathrm{i} (xa + yp) + \frac{2\pi \mathrm{i}}{N} xp \sum_{j=0}^{N-1} j.
\]
Using $\sum_{j=0}^{N-1} j = \frac{N(N-1)}{2}$, the second term becomes $\pi \mathrm{i} xp (N-1)$. Since $N$ is even, $N-1$ is odd; therefore $\exp(\pi \mathrm{i} xp(N-1)) = (-1)^{xp}$. The term $\exp(2\pi \mathrm{i} (xa+yp)) = 1$.
Thus the total phase is
\begin{equation}
\Lambda_N(a, p) = (-1)^p (-1)^{xp} = (-1)^{p(1+x)}.
\end{equation}
For $\tilde{t}^N = (\id, \mathbf{1})$, we require $\Lambda_N(a, p) = 1$ for all $p$. This holds if and only if $1+x$ is even, i.e., $x \equiv 1 \pmod 2$.
\end{proof}

\begin{lemma}\label{lem:modular-constraint}
Let $N=2^k$ with $k \geq 1$. Choose lifts $\tilde t=(t,\lambda_t)$ and $\tilde s=(s,\lambda_s)$ as in \eqref{eq:lift-t}--\eqref{eq:lift-s} with parameters $(x,y)$ and $(z,w)$, respectively.
If the modular relation $(\tilde s\tilde t)^3=\tilde s^{\,2}$ holds in the Clifford group, then
\[
2x\equiv 0\pmod N.
\]
\end{lemma}

\begin{proof}
We organize the computation in two steps: first we verify the relation for a convenient pair of reference lifts, and then we determine how the relation changes after twisting by arbitrary linear characters.

\smallskip
\noindent\textbf{Step 1: The reference lifts satisfy $(\tilde s_0\tilde t_0)^3=\tilde s_0^{\,2}$.}
Fix the particular solutions
\[
\lambda_t^{(0)}(a,p)\coloneqq\exp\!\Big(\frac{\pi \mathrm{i}}{N}p^2\Big),\qquad
\lambda_s^{(0)}(a,p)\coloneqq\exp\!\Big(-\frac{2\pi \mathrm{i}}{N}ap\Big),
\]
and set $\tilde t_0\coloneqq(t,\lambda_t^{(0)})$,
$\tilde s_0\coloneqq(s,\lambda_s^{(0)})$.

First,
\[
\tilde s_0^{\,2}=(s^2,(\lambda_s^{(0)})^{s}\lambda_s^{(0)})=(-I,\mathbf{1}),
\]
because $(\lambda_s^{(0)})^{s}\lambda_s^{(0)}\equiv 1$. Set
$\tilde g_0\coloneqq\tilde s_0\tilde t_0=(st,\lambda_{g_0})$, where
\[
\lambda_{g_0}(a,p)=(\lambda_s^{(0)})^{t}(a,p)\lambda_t^{(0)}(a,p)
=\exp\!\Big(-\frac{2\pi \mathrm{i}}{N}(a+p)p\Big)\exp\!\Big(\frac{\pi \mathrm{i}}{N}p^2\Big)
=\exp\!\Big(-\frac{2\pi \mathrm{i}}{N}ap\Big)\exp\!\Big(-\frac{\pi \mathrm{i}}{N}p^2\Big).
\]
Since $g\coloneqq st$ satisfies $g^3=-I=s^2$ in $\SL(2,\Z_N)$, it remains to compare the phase functions.  The twisted product gives
\[
(\tilde s_0\tilde t_0)^3=\tilde g_0^{\,3}=(g^3,\Lambda_0),\qquad
\Lambda_0(u)=\lambda_{g_0}(g^2u)\lambda_{g_0}(gu)\lambda_{g_0}(u),
\]
Since $g(a,p)=(-p,a+p)$ and $g^2(a,p)=(-a-p,a)$, direct substitution
gives
\[
\Lambda_0(a,p)
=\exp\!\left(
\frac{\pi \mathrm{i}}{N}
\bigl((a^2+2ap)+(p^2-a^2)+(-2ap-p^2)\bigr)
\right)
=1.
\]
Therefore
\[
(\tilde s_0\tilde t_0)^3=(-I,\mathbf{1})=\tilde s_0^{\,2}.
\]

\smallskip
\noindent\textbf{Step 2: Reduction to a linear-character computation.}
We now compare arbitrary lifts with the reference lifts from Step~1.
Let $\xi_v\coloneqq(\id,\eta_v)$, where
\[
\eta_v(u)\coloneqq\exp\!\left(\frac{2\pi \mathrm{i}}{N}\langle v,u\rangle\right),
\qquad
\langle v,u\rangle \coloneqq v_1u_1+v_2u_2 \pmod N
\]
for $v=(v_1,v_2)$ and $u=(u_1,u_2)$ in $\Z_N^2$.
All vector and matrix operations in the remainder of the proof are
performed over $\Z_N$.
For any lift $\widetilde H=(H,\lambda)$, the twisted product
gives
\begin{equation}\label{eq:char-conj}
\xi_v\widetilde H=\widetilde H\xi_{H^Tv}.
\end{equation}

Write the general lifts as
\[
\tilde t=\tilde t_0\xi_{v_t},\qquad
\tilde s=\tilde s_0\xi_{v_s},
\]
where $v_t=\binom{x}{y}$ and $v_s=\binom{z}{w}$.
Using \eqref{eq:char-conj}, we obtain
\[
\tilde s\tilde t=\tilde g_0\xi_r,
\qquad
r=t^Tv_s+v_t
=\begin{pmatrix}1&0\\1&1\end{pmatrix}\binom{z}{w}+\binom{x}{y}
=\binom{x+z}{z+w+y}.
\]
Set $g\coloneqq st$; then $g^T=\begin{pmatrix}0&1\\-1&1\end{pmatrix}$.  Hence
\[
(\tilde s\tilde t)^3
=\tilde g_0^3\xi_{((g^T)^2+g^T+I)r},
\qquad
\tilde s^2=\tilde s_0^2\xi_{(s^T+I)v_s}.
\]
By Step~1, $\tilde g_0^3=\tilde s_0^2$. Since
$v\mapsto\eta_v$ is injective, the modular relation implies
\[
\bigl((g^T)^2+g^T+I\bigr)r=(s^T+I)v_s.
\]
The two sides are
\[
\begin{pmatrix}0&2\\-2&2\end{pmatrix}
\binom{x+z}{z+w+y}
=\binom{2(z+w+y)}{2(w+y-x)},
\qquad
\begin{pmatrix}1&1\\-1&1\end{pmatrix}
\binom{z}{w}
=\binom{z+w}{w-z}.
\]
Equating their components gives
\[
z+w+2y\equiv0\pmod N,
\qquad
z+w+2y-2x\equiv0\pmod N.
\]
Subtracting the two congruences yields
\[
2x\equiv0\pmod N,
\]
as claimed.
\end{proof}

\begin{theorem}\label{thm:nonsplit-cyclic}
Let $N=2^k$ with $k\ge 2$. The Clifford extension
\[
1\longrightarrow V_{\Z_N}\longrightarrow C(\Z_N)\longrightarrow \SL(2,\Z_N)\longrightarrow 1
\]
does not split.
\end{theorem}

\begin{proof}
Suppose the extension splits via a homomorphism $\sigma\colon \SL(2,\Z_N) \to C(\Z_N)$. Let $\tilde{t} = \sigma(t)$ and $\tilde{s} = \sigma(s)$ be the images of the generators, with $\tilde{t} = (t, \lambda_t)$ parametrized by $(x,y) \in \Z_N^2$ as in \eqref{eq:lift-t}.

Since $\sigma$ is a homomorphism, its values on $s$ and $t$ must
preserve the relations $t^N=I$ and $(st)^3=s^2$.
Lemmas~\ref{lem:parity-constraint} and~\ref{lem:modular-constraint}
impose incompatible conditions on the parameter $x$:
\begin{itemize}
\item The relation $t^N = I$ implies $\tilde{t}^N = (\id, \mathbf{1})$, hence by Lemma~\ref{lem:parity-constraint}, $x$ must be odd.
\item The relation $(st)^3 = s^2$ implies $(\tilde{s}\tilde{t})^3 = \tilde{s}^2$, thus by Lemma~\ref{lem:modular-constraint}, $2x \equiv 0 \pmod{N}$.
\end{itemize}

For $N = 2^k$ with $k \geq 2$, the condition $2x \equiv 0 \pmod{N}$ implies $x \in \{0, N/2\}$. Since $N/2 = 2^{k-1}$ is even for $k \geq 2$, both possible values of $x$ are even. This contradicts the requirement that $x$ be odd.

Therefore, no splitting homomorphism exists.
\end{proof}

\begin{remark}
No order condition on the lift $\tilde{s}$ is used in the proof.
\end{remark}
\section{Symplectic obstruction in the elementary abelian case}\label{sec:symplectic-obstruction}

For $A=(\Z_2)^n$ with $n\geq2$, we identify the Clifford extension with
the symplectic automorphism extension of the central product studied by
Griess~\cite{Griess1973}.  Its middle term consists of the
automorphisms that fix the central copy of $\Z_4$ pointwise and is
distinct from the full automorphism group of the extraspecial factor.
Griess' results then imply that the Clifford extension does not split.

\subsection{The Pauli group and its automorphisms}

Fix $A = (\Z_2)^n$ with $V_A = A \oplus \widehat{A} \cong (\F_2)^{2n}$.
Let
\[
\overline{\omega}_A\colon V_A\times V_A\longrightarrow\F_2
\]
be the alternating bilinear form determined by
$\omega_A(u,v)=(-1)^{\overline{\omega}_A(u,v)}$. Thus $\Sp(V_A)$ is
the symplectic group of $(V_A,\overline{\omega}_A)$; after choosing a
symplectic basis,
\[
\Sp(V_A)\cong\Sp(2n,\F_2).
\]
Set $\mu_4\coloneqq\langle \mathrm{i}\rangle=\{1,\mathrm{i},-1,-\mathrm{i}\}$.  The $n$-qubit
\emph{Pauli group} $P_n$ is the finite subgroup of $U(\HH)$ generated
by the Weyl operators and the scalar $\mathrm{i}I$:
\begin{equation}\label{eq:pauli-def}
P_n \coloneqq \langle\, W_u,\, \mathrm{i}I \mid u \in V_A \,\rangle.
\end{equation}
Explicitly, $P_n = \{ z W_u \mid z \in \mu_4,\, u \in V_A \}$ has
order $2^{2n+2}$. Its center is
\[
Z(P_n)=Y\coloneqq\mu_4 I=\langle \mathrm{i}I\rangle\cong\Z_4,
\]
and the map
\[
P_n/Y\longrightarrow V_A,\qquad zW_uY\longmapsto u,
\]
is an isomorphism.

To identify the extraspecial factor $E_n$ in $P_n$, choose a basis
$e_1,\dots,e_n$ of $A$ and its dual basis
$e_1^\vee,\dots,e_n^\vee$ of $\widehat A$.  Set
\[
X_j\coloneqq W_{(e_j,0)},
\qquad
Z_j\coloneqq W_{(0,e_j^\vee)}.
\]
The subgroup $D_j=\langle X_j,Z_j\rangle$ is isomorphic to $D_8$, the
dihedral group of order $8$:
$X_j$ and $Z_j$ are involutions and their commutator is $-I$.
Subgroups with different indices commute, and their common center is
$\langle-I\rangle$.  Hence
\[
E_n\coloneqq\langle\, W_u,\,-I \mid u\in V_A\,\rangle
=D_1\circ\cdots\circ D_n,
\]
where $\circ$ denotes the central product obtained by amalgamating the
common central subgroup $\langle-I\rangle$.  The group $E_n$ is an
extraspecial $2$-group of order $2^{2n+1}$, with center and
commutator subgroup both equal to $\langle-I\rangle$.  It is of plus
(positive) type and is exactly the group denoted $E_n$ by
Griess~\cite[pp.~405--406]{Griess1973}.

From the generators in \eqref{eq:pauli-def} we have
\[
P_n=E_nY,
\qquad
E_n\cap Y=\langle-I\rangle,
\qquad
[E_n,Y]=1.
\]
Thus $P_n$ is the internal central product
\[
P_n=E_n\circ Y,
\]
formed by amalgamating the common subgroup $\langle-I\rangle$.
Griess' results apply to the central product $P_n=E_n\circ Y$.

We use the subgroup considered by Griess~\cite{Griess1973}:
\[
\Aut_Y(P_n) \coloneqq \{ \varphi \in \Aut(P_n) \mid \varphi|_Y = \id_Y \},
\]
of automorphisms that fix $Y$ pointwise. In Griess' notation, this
group is denoted by $A(E_n \circ Y)$. Every such automorphism induces
a well-defined automorphism of the quotient
\[
P_n/Y \cong V_A.
\]
Because $\varphi$ fixes the center, it preserves the commutator form
$\overline{\omega}_A$, and hence the induced map on $V_A$ lies in
$\Sp(V_A)\cong\Sp(2n,\F_2)$. Griess identifies $V_A$ with
$\Inn(P_n)$ via the map~\cite[pp.~403--404]{Griess1973}
\[
v\longmapsto\operatorname{Ad}(W_v),\qquad
\operatorname{Ad}(W_v)(W_u)=\omega_A(v,u)W_u.
\]
The second exact sequence in Griess' Corollary~2 therefore takes the
form
\begin{equation}\label{eq:autZ-ses}
1 \longrightarrow V_A \longrightarrow \Aut_Y(P_n) \xrightarrow{\;\rho\;} \Sp(V_A) \longrightarrow 1,
\end{equation}
as stated in~\cite[Corollary~2, p.~407]{Griess1973}.

\subsection{Identification with the Clifford extension}

We compare the middle terms of \eqref{eq:autZ-ses} and the Clifford
extension through their descriptions by pairs $(T,\lambda)$.

Let $T \in \Sp(V_A)$. An automorphism $\varphi \in \Aut_Y(P_n)$ lifting $T$ must satisfy $\varphi(W_u) = \lambda(u) W_{Tu}$ for some function $\lambda\colon V_A \to \mu_4$. Imposing that $\varphi$ preserves multiplication yields the coboundary condition
\begin{equation}\label{eq:aut-cob}
\frac{\lambda(u+v)}{\lambda(u)\lambda(v)} = \frac{\beta_A(Tu, Tv)}{\beta_A(u, v)}.
\end{equation}
Condition~\eqref{eq:aut-cob} agrees with
\eqref{eq:coboundary-condition} in Theorem~\ref{thm:clifford-structure},
except that $\lambda$ is initially required to take values in $\mu_4$
rather than $U(1)$. For elementary abelian $2$-groups the right-hand
side of~\eqref{eq:aut-cob} lies in $\{\pm1\}$. Moreover, any
$U(1)$-valued solution of \eqref{eq:coboundary-condition} has
$\lambda(0)=1$, and substituting $u=v$ gives
$\lambda(u)^2\in\{\pm1\}$. Hence every such solution has values in
$\mu_4$, so the two conditions coincide.

\begin{proposition}\label{prop:ses-ident}
There is an isomorphism of short exact sequences
\[
\begin{tikzcd}[column sep=large]
1 \arrow[r] & V_A \arrow[r] \arrow[d, equals] & C(A) \arrow[r, "\pi"] \arrow[d, "\Theta", "\cong"'] & \Sp(V_A) \arrow[r] \arrow[d, equals] & 1 \\
1 \arrow[r] & V_A \arrow[r] & \Aut_Y(P_n) \arrow[r, "\rho"] & \Sp(V_A) \arrow[r] & 1
\end{tikzcd}
\]
identifying both middle terms with pairs $(T, \lambda)$ satisfying~\eqref{eq:aut-cob}.
\end{proposition}

\begin{proof}
Let $(T,\lambda)\in C(A)\cong\Ps(A)$.  Since $V_A$ has exponent
$2$, substituting $u=v$ in \eqref{eq:aut-cob} shows that
\[
\lambda(u)^2
=\frac{\beta_A(u,u)}{\beta_A(Tu,Tu)}\in\{\pm1\}.
\]
Thus $\lambda$ takes values in $\mu_4$.  Define
\begin{equation}\label{eq:theta-definition}
\Theta(T,\lambda)(zW_u)
\coloneqq z\lambda(u)W_{Tu},
\qquad z\in\mu_4.
\end{equation}
Condition~\eqref{eq:aut-cob} shows that the map defined in
\eqref{eq:theta-definition} preserves products and therefore defines
an automorphism of $P_n$ fixing $Y$ pointwise.

Conversely, every $\varphi\in\Aut_Y(P_n)$ induces a unique
$T\in\Sp(V_A)$ and has a unique expression
$\varphi(W_u)=\lambda(u)W_{Tu}$ with
$\lambda(u)\in\mu_4$.  Preservation of multiplication gives
\eqref{eq:aut-cob}, so $(T,\lambda)\in\Ps(A)$.  These two constructions
are inverse.  Composition on either side is governed by
\[
(T,\lambda)(S,\mu)=(TS,\lambda^S\mu),
\]
so $\Theta$ is a group isomorphism.

It preserves the symplectic component.  On the kernel, conjugation by
$W_v$ acts as
$W_u\mapsto\omega_A(v,u)W_u$, so the identification with inner
automorphisms is precisely the map $\kappa_A$ used in the Clifford
extension.  Hence the two squares determined by $\Theta$, the identity
map on $V_A$, and the identity map on $\Sp(V_A)$ commute, so the diagram
is an isomorphism of extensions.
\end{proof}
\subsection{The rank-one exception and nonsplitting for \texorpdfstring{$n \geq 2$}{n >= 2}}

Proposition~\ref{prop:ses-ident} transfers the splitting problem for
the Clifford extension to the automorphism
extension~\eqref{eq:autZ-ses}, to which Griess' nonsplitting results
apply.

\begin{theorem}\label{thm:elem-abelian-nonsplit}
Let $A = (\Z_2)^n$ with $n\ge 1$. The Clifford extension
\[
1 \longrightarrow V_A \longrightarrow C(A) \longrightarrow \Sp(V_A) \longrightarrow 1
\]
splits if and only if $n = 1$.
\end{theorem}

\begin{proof}
For $n = 1$, we have
$\Sp(V_A)\cong\Sp(2,\F_2)\cong\mathfrak{S}_3$.  We give an explicit
section.  With $s,t$ as in \eqref{eq:cyclic-generators} for $N=2$, set
\[
\lambda_t(a,p)=\mathrm{i}^p(-1)^a,
\qquad
\lambda_s(a,p)=(-1)^{ap}.
\]
These are the lifts \eqref{eq:lift-t}--\eqref{eq:lift-s} with
$(x,y)=(1,0)$ and $(z,w)=(0,0)$.  Direct substitution in the twisted
product gives
\[
\widetilde t^2=\widetilde s^2
=(\widetilde s\widetilde t)^3=(\id,\mathbf{1}).
\]
Since
$\mathfrak{S}_3=\langle s,t\mid s^2=t^2=(st)^3=1\rangle$, these lifts
define a homomorphic section.

For $n \geq 2$, Griess' nonsplitting results apply to the automorphism
extension~\eqref{eq:autZ-ses} of $P_n=E_n\circ Y$.  The second sequence in
Griess' Corollary~2 states that
\[
1 \longrightarrow \Inn(P_n) \longrightarrow \Aut_Y(P_n) \longrightarrow \Sp(2n,\F_2) \longrightarrow 1
\]
is nonsplit for $n \geq 3$~\cite[Corollary~2,
p.~407]{Griess1973}.  Griess' Corollary~3 proves that these same exact
sequences are nonsplit for $n=2$~\cite[Corollary~3,
p.~408]{Griess1973}.  Proposition~\ref{prop:ses-ident} transports these
conclusions to the Clifford extension.
\end{proof}

\section{Proof of the splitting criterion}\label{sec:splitting-criterion}

It remains to combine the reduction to the $2$-primary part with
Theorems~\ref{thm:nonsplit-cyclic} and~
\ref{thm:elem-abelian-nonsplit}.

\begin{theorem}
Let $A$ be a finite abelian group and let $V_A=A\oplus \widehat A$.
The Clifford extension
\begin{equation}\label{eq:main-ext}
1 \longrightarrow V_A \longrightarrow C(A) \longrightarrow \Sp(V_A) \longrightarrow 1
\end{equation}
splits if and only if $4 \nmid |A|$.
\end{theorem}

\begin{proof}
Write the primary decomposition
\[
A \;\cong\; A_{\mathrm{odd}}\oplus A_2,
\]
where $|A_{\mathrm{odd}}|$ is odd and $|A_2|$ is a power of $2$.
By Proposition~\ref{prop:coprime-decomp}, the Clifford extension for
$A$ is the product of the extensions for $A_{\mathrm{odd}}$ and $A_2$.
The extension for $A_{\mathrm{odd}}$ splits by
Proposition~\ref{prop:odd_case_vanish}; hence
\eqref{eq:main-ext} splits if and only if the $2$-primary extension for
$A_2$ splits.
Thus it suffices to determine when the extension splits for $A_2$.

\medskip
\noindent\textit{Step 1: the splitting cases.}
If $A_2=1$, then $A$ has odd order and the extension splits by Proposition~\ref{prop:odd_case_vanish}.
If $A_2\cong \Z_2$, then the extension splits by Theorem~\ref{thm:elem-abelian-nonsplit}, which treats the exceptional rank-one elementary abelian case.

\medskip
\noindent\textit{Step 2: the nonsplitting cases.}
Assume $|A_2| \geq 4$.
Write
\[
A_2 \;\cong\; \Z_{2^{k_1}}\oplus \cdots \oplus \Z_{2^{k_m}},
\qquad k_1\ge \cdots \ge k_m \ge 1.
\]
There are then exactly two possibilities:

\smallskip
\noindent\emph{(A) $k_1\ge2$.}
Then $\Z_{2^{k_1}}$ is a direct summand of $A_2$ of order at least
$4$.
By Theorem~\ref{thm:nonsplit-cyclic}, the extension for $\Z_{2^{k_1}}$ does not split.
By Theorem~\ref{thm:embedding}, nonsplitting extends to any group containing
$\Z_{2^{k_1}}$ as a direct summand; therefore $[\cO_{A_2}]\neq 0$.

\smallskip
\noindent\emph{(B) $k_1=1$.}
Then $k_i=1$ for all $i$, and $|A_2|\ge4$ implies $m\ge2$. Thus
$A_2\cong(\Z_2)^m$, and $(\Z_2)^2$ is a direct summand of $A_2$.
By Theorem~\ref{thm:elem-abelian-nonsplit}, the extension for $(\Z_2)^2$ does not split,
hence $[\cO_{(\Z_2)^2}]\neq 0$.
Applying Theorem~\ref{thm:embedding} yields $[\cO_{A_2}]\neq 0$ for all $m\ge2$.

\medskip
Thus \eqref{eq:main-ext} splits exactly when $A_2=1$ or
$A_2\cong\Z_2$, and these are the cases in which $4\nmid|A|$.
\end{proof}

\section*{Acknowledgments}
The author thanks both referees for their careful reading and constructive comments.
The author was partially supported by Grant INV-2025-213-3452 from the School of Science of Universidad de los Andes.

\section*{Conflict of interest}
The author declares that he has no conflict of interest.

\bibliographystyle{abbrvurl}
\bibliography{260516-Galindo}

@article{BravyiKitaev2005,
  author  = {Bravyi, Sergey and Kitaev, Alexei},
  title   = {Universal quantum computation with ideal {Clifford} gates and noisy ancillas},
  journal = {Phys Rev A},
  volume  = {71},
  number  = {2},
  pages   = {022316},
  year    = {2005},
  doi     = {10.1103/PhysRevA.71.022316},
}

@article{KT23,
  title = {{Clifford} group is not a semidirect product in dimensions {$N$} divisible by four},
  author = {Korbel{\'a}{\v{r}}, Miroslav and Tolar, Ji{\v{r}}{\'i}},
  journal = {J Phys A Math Theor},
  year = {2023},
  volume = {56},
  number = {27},
  pages = {275304},
  doi = {10.1088/1751-8121/acd891}
}

@article{Griess1973,
  title = {Automorphisms of extra special groups and nonvanishing degree 2 cohomology},
  author = {Griess, Jr., Robert L.},
  journal = {Pac J Math},
  volume = {48},
  number = {2},
  pages = {403--422},
  year = {1973},
  doi = {10.2140/pjm.1973.48.403}
}

@article{Weil1964,
  title = {Sur certains groupes d'opérateurs unitaires},
  author = {Weil, André},
  journal = {Acta Math},
  volume = {111},
  pages = {143--211},
  year = {1964},
  doi = {10.1007/BF02391012}
}

@book{Brown1982,
  title = {Cohomology of Groups},
  author = {Brown, Kenneth S.},
  series = {Graduate Texts in Mathematics},
  volume = {87},
  year = {1982},
  publisher = {Springer-Verlag},
  address = {New York},
  doi = {10.1007/978-1-4684-9327-6}
}

@phdthesis{Gottesman1997,
  author = {Gottesman, Daniel},
  title = {Stabilizer codes and quantum error correction},
  school = {California Institute of Technology},
  address = {Pasadena},
  year = {1997},
  note = {arXiv:quant-ph/9705052},
  doi = {10.7907/RZR7-DT72}
}

@article{Gottesman1998,
  title = {Theory of fault-tolerant quantum computation},
  author = {Gottesman, Daniel},
  journal = {Phys Rev A},
  volume = {57},
  number = {1},
  pages = {127--137},
  year = {1998},
  doi = {10.1103/PhysRevA.57.127}
}

@article{AaronsonGottesman2004,
  author = {Aaronson, Scott and Gottesman, Daniel},
  title = {Improved simulation of stabilizer circuits},
  journal = {Phys Rev A},
  volume = {70},
  number = {5},
  pages = {052328},
  year = {2004},
  doi = {10.1103/PhysRevA.70.052328}
}

@article{Appleby2005,
  author = {Appleby, D. M.},
  title = {Symmetric informationally complete--positive operator valued measures and the extended {Clifford} group},
  journal = {J Math Phys},
  volume = {46},
  number = {5},
  pages = {052107},
  year = {2005},
  doi = {10.1063/1.1896384}
}

@article{Gross2006,
  author = {Gross, David},
  title = {Hudson's theorem for finite-dimensional quantum systems},
  journal = {J Math Phys},
  volume = {47},
  number = {12},
  pages = {122107},
  year = {2006},
  doi = {10.1063/1.2393152}
}

@article{Hostens2005,
  author = {Hostens, Erik and Dehaene, Jeroen and De Moor, Bart},
  title = {Stabilizer states and {Clifford} operations for systems of arbitrary dimensions and modular arithmetic},
  journal = {Phys Rev A},
  volume = {71},
  number = {4},
  pages = {042315},
  year = {2005},
  doi = {10.1103/PhysRevA.71.042315}
}

@article{MosesEtAl2026,
  author = {Moses, Milo and Horecki, Jacek and Deka, Konrad and Tu{\l}owiecki, Jan},
  title = {{Clifford} circuits on qudits of mixed dimension},
  journal = {Quantum Inf Comput},
  volume = {26},
  number = {1},
  pages = {180--191},
  year = {2026},
  doi = {10.2478/qic-2026-0009}
}

@article{Knill2008,
  author = {Knill, E. and Leibfried, D. and Reichle, R. and Britton, J. and Blakestad, R. B. and Jost, J. D. and Langer, C. and Ozeri, R. and Seidelin, S. and Wineland, D. J.},
  title = {Randomized benchmarking of quantum gates},
  journal = {Phys Rev A},
  volume = {77},
  number = {1},
  pages = {012307},
  year = {2008},
  doi = {10.1103/PhysRevA.77.012307}
}

@article{Dankert2009,
  author = {Dankert, Christoph and Cleve, Richard and Emerson, Joseph and Livine, Etera},
  title = {Exact and approximate unitary 2-designs and their application to fidelity estimation},
  journal = {Phys Rev A},
  volume = {80},
  number = {1},
  pages = {012304},
  year = {2009},
  doi = {10.1103/PhysRevA.80.012304}
}

@article{Magesan2011,
  author = {Magesan, Easwar and Gambetta, J. M. and Emerson, Joseph},
  title = {Scalable and robust randomized benchmarking of quantum processes},
  journal = {Phys Rev Lett},
  volume = {106},
  number = {18},
  pages = {180504},
  year = {2011},
  doi = {10.1103/PhysRevLett.106.180504}
}

@article{Tambara2000,
  title = {Representations of tensor categories with fusion rules of self-duality for abelian groups},
  author = {Tambara, Daisuke},
  journal = {Isr J Math},
  volume = {118},
  number = {1},
  pages = {29--60},
  year = {2000},
  doi = {10.1007/BF02803515}
}

@article{Mackey1949,
  author = {Mackey, George W.},
  title = {A theorem of {Stone} and von {Neumann}},
  journal = {Duke Math J},
  volume = {16},
  number = {2},
  pages = {313--326},
  year = {1949},
  doi = {10.1215/S0012-7094-49-01631-2}
}

@article{Prasad2011,
  author = {Prasad, Amritanshu},
  title = {An easy proof of the {Stone--von Neumann--Mackey} theorem},
  journal = {Expo Math},
  volume = {29},
  number = {1},
  pages = {110--118},
  year = {2011},
  doi = {10.1016/j.exmath.2010.06.001}
}

@article{KaiblingerNeuhauser2009,
  author = {Kaiblinger, Norbert and Neuhauser, Markus},
  title = {Metaplectic operators for finite abelian groups and {$\mathbb{R}^{d}$}},
  journal = {Indag Math},
  volume = {20},
  number = {2},
  pages = {233--246},
  year = {2009},
  doi = {10.1016/S0019-3577(09)80011-6}
}

@article{DuttaPrasad2015,
  author = {Dutta, Kunal and Prasad, Amritanshu},
  title = {Combinatorics of finite abelian groups and {Weil} representations},
  journal = {Pac J Math},
  volume = {275},
  number = {2},
  pages = {295--324},
  year = {2015},
  doi = {10.2140/pjm.2015.275.295}
}

@article{Galindo2017,
  author = {Galindo, C\'esar},
  title = {Isocategorical groups and their {Weil} representations},
  journal = {Trans Amer Math Soc},
  volume = {369},
  number = {11},
  pages = {7935--7960},
  year = {2017},
  doi = {10.1090/tran/6919}
}

@article{Mackey1958,
  author = {Mackey, George W.},
  title = {Unitary representations of group extensions. {I}},
  journal = {Acta Math},
  volume = {99},
  pages = {265--311},
  year = {1958},
  doi = {10.1007/BF02392428}
}

@article{BoltRoomWall1961a,
  author = {Bolt, Beverley and Room, T. G. and Wall, G. E.},
  title = {On the {Clifford} collineation, transform and similarity groups. {I}},
  journal = {J Aust Math Soc},
  volume = {2},
  number = {1},
  pages = {60--79},
  year = {1961},
  doi = {10.1017/S1446788700026379}
}

@article{BoltRoomWall1961b,
  author = {Bolt, Beverley and Room, T. G. and Wall, G. E.},
  title = {On the {Clifford} collineation, transform and similarity groups. {II}},
  journal = {J Aust Math Soc},
  volume = {2},
  number = {1},
  pages = {80--96},
  year = {1961},
  doi = {10.1017/S1446788700026380}
}

@misc{KorbelarTolar2026,
  author = {Korbel{\'a}{\v{r}}, Miroslav and Tolar, Ji{\v{r}}{\'i}},
  title = {Structure of {Clifford} groups of composite finite quantum systems},
  year = {2026},
  note = {arXiv:2606.08215 [math-ph]}
}
\end{document}